\documentclass[12pt, notitlepage]{amsart}
\usepackage{latexsym, amsfonts, amsmath, amssymb, amsthm, cite}
\usepackage{graphicx}
\usepackage[linktocpage=true,colorlinks,citecolor=blue,linkcolor=blue,urlcolor=blue]{hyperref}

\usepackage{amssymb}
\usepackage{cite}
\usepackage{amsmath}
\usepackage{latexsym}
\usepackage{amscd}
\usepackage{tikz}
\usepackage{amsthm}
\usepackage{mathrsfs}
\usepackage{url}
\usepackage[utf8]{inputenc}
\usepackage[english]{babel}
\usepackage{amsfonts}

\numberwithin{equation}{section}
\def\Z{\mathbb Z}
\def\F{\mathbb F}
\def\C{\mathbb C}

\def\Q{\mathbb Q}

\newtheorem{theorem}{Theorem}[section]
\newtheorem{lemma}[theorem]{Lemma}

\theoremstyle{remark}

\theoremstyle{definition}

\theoremstyle{remark}
\newtheorem{example}[theorem]{Example}

\numberwithin{equation}{section}

\begin{document}
\title{A polynomial with a root mod $p$ for every $p$ has a real root}

\author{Rodrigo Angelo}
\address{Department of Mathematics, Stanford University, Stanford, CA 94305}
\email{rsangelo@stanford.edu}

\author{Max Wenqiang Xu}
\address{Department of Mathematics, Stanford University, Stanford, CA 94305}
\email{maxxu@stanford.edu}

\maketitle
\begin{abstract}
We prove that if a polynomial with rational coefficients has a root mod $p$ for every large prime $p$, then it has a real root. As an application, we show that the primes can't be covered by finitely many positive definite binary quadratic forms.
\end{abstract}

\section{Introduction}
It is known that if a polynomial irreducible over $\mathbb{Q}[x]$ has a root mod $p$ for every prime $p$, then it has a rational root (and is therefore linear) \cite{Lenstra}.   This is subtler if the irreducibility condition is dropped, and there exist polynomials e.g.,
\begin{equation}\label{examples}
    \begin{gathered}
(x^2+1)(x^2+2)(x^2-2) ~\text{and}\\
(x^2+x+1)(x^3-2)
\end{gathered}
\end{equation}
\cite{Brandl} that have a root mod $p$ for every prime $p$, yet no rational root. The first polynomial always has a root mod $p$ because of the identity of Legendre symbols $\left(\frac{-1}{p}\right)\left(\frac{-2}{p}\right) = \left(\frac{2}{p}\right)$, which ensures that always at least one of $-1$, $-2$ or $2$ is a square mod $p$. In the second, due to quadratic reciprocity $x^2+x+1$ has a root unless $p \equiv 2$ mod $3$, but in that case $x = 2^{\frac{2p-1}{3}}$ is a solution to $x^3 \equiv 2$ mod $3$. 

Notice however that these two polynomials still have a real root. We show this holds generally.

 \begin{theorem}\label{real root}
Let $f\in \mathbb{Q}[x]$ be a polynomial that has a root mod $ p$ for every large prime $p$. Then $f$ has a real root.
\end{theorem}

This is an application of Chebotarev's density theorem - the proof itself is fairly straightforward, so much of the paper is dedicated to overviewing the background to Chebotarev's theorem, which is not straightforward. There is a wide literature about polynomials with a root mod $p$ for every $p$. For example \cite{Bilu} describes a criteria in terms of Galois theory to decide if a polynomial has this property. The papers \cite{Mishra, Spearman} show that families of polynomials similar to these two examples have a solution mod $n$ for every integer $n$, and \cite{Sonn} looks gives criteria for having a root in $\Q_p$ for every $p$. Products of two irreducible factors with this property are studied in \cite{Brandl}. We expand on this literature, noticing the common thread that these diverse constructions all have a real root. 

Our result was motivated by a problem on representing primes with positive definite quadratic forms. A binary quadratic form is an expression of the form $g(x,y) = ax^{2} +bxy +cy^{2} \in \Z[x,y] $, and its discriminant is  $-D =  b^{2} -4ac$ (we also assume $a>0$ without losing much generality). We call $f$ positive definite if $-D<0$, in which case $g(x,y) > 0$ for all $x,y \neq 0$.  The problem of which primes can be represented by a form $f$, that is, written as $p = g(m,n)$ for some integers $m$ and $n$ has been extensively studied. A necessary condition is that the equation $g(x,y) \equiv 0$ mod $p$ needs to have a non-trivial solution. This requires $-D$ to be a square mod $p$. In certain cases (when the class number of the discriminant $-D$ is $1$) this condition is sufficient, but in general it is not. A fantastic study of the representation of primes by quadratic forms can be found at \cite{cox}.  

A natural task is to look for quadratic forms that together represent all the primes. For instance $x^2+y^2, x^2+2y^2$ and $x^2-2y^2$ together cover all primes. This is because a sufficient criterion for $p$ being covered by each of these forms is that $-1, -2$ or $2$ is a square mod $p$ respectively (for each of these forms the requirement ``$-D$ is a square" is in fact sufficient - which stems from $-4$, $-8$ and $8$ each being quadratic discriminants of class number 1), and as we pointed before one of $-1, -2$ or $2$ is always a square mod $p$. A different example is $x^2+y^2$, $x^2-5y^2$, $x^2+5y^2$ and  $2x^2+2xy+3y^2$, which together cover all primes. These have discriminants $-4, 20, -20$ and $-20$, one of which is always being a square mod $p$. For $-20$ we need to include two forms to cover all the classes, due to the class number of the discriminant $-20$ being $2$ (see page 32 of \cite{cox} for an explanation that $x^2+5y^2$ and  $2x^2+2xy+3y^2$ together cover every prime for which $-20$ is a square mod $p$). For any odd cardinality set of discriminants whose product is a perfect square (this guarantees one of the discriminants is always a square mod $p$), a collection of forms achieving all $h(-D)$ classes of each discriminant $-D$ (which guarantees one of those forms will cover any prime $p$ for which that $-D$ is a square mod $p$) will cover every large prime. These constructions we provided always include a form such as $x^2-2y^2$ or $x^2-5y^2$ with positive discriminant - one can ask whether it is possible to achieve this task with only positive definite forms. We answer this negatively.

\begin{theorem}\label{positive}
There is no finite set of positive definite binary quadratic forms such that for every large prime $p$, at least one of the forms has a non-trivial solution to $g(x,y) \equiv 0$ mod $p$.
\end{theorem}

By non-trivial we mean a solution other than $(x,y) = (0,0)$ mod $p$. In particular there is no finite set of positive definite binary quadratic forms together representing all the primes, since each solution to $g(x,y)= p$ provides a non-trivial solution of $g(x,y) \equiv 0$ mod $p$ (in a trivial solution $p^2$ divides $g(x,y)$).  Will Jagy independently asked this question on Mathoverflow, and Lucia \cite{Lucia} gave a proof based on Dirichlet's theorem and density considerations. We offer a proof using our main result on polynomials with a root mod $p$ for every $p$.

\subsection*{Acknowledgement} We thank the anonymous referees, whose detailed comments helped us refine our proofs and exposition.

\section{Background and examples}

Let us begin recalling the definition of the Frobenius automorphism of Galois extensions of $\Q$, and the statement of Chebotarev's density theorem. 

Given a polynomial $f \in \mathbb{Q}[x]$, let $K \subset \mathbb{C}$ be the splitting field of $f$ over $\mathbb{Q}$ (the smallest extension of $\Q$ containing all roots of $f$), with Galois group $G = Aut(K/\mathbb{Q})$, and let $O_K$ be the ring of algebraic integers in $K$. For a prime $p \in \Z$, let $\mathcal{P}$ be a maximal ideal in $O_K$ containing the ideal $pO_K$ (it is a general algebraic fact that every proper ideal in a ring is inside at least one maximal ideal - in the language of algebraic number theory this called a prime ideal over $K$ above $p$). It is handy to know that $\mathcal{P}\cap \mathbb{Z} = p \mathbb{Z}$ for such an ideal. 

If $p$ does not divide the discriminant of $K/\Q$, there always exists a unique automorphism $\pi \in G$ with the property $\pi(x) \equiv x^p$ mod $\mathcal{P}$ for every $x \in O_K$ (meaning $\pi(x) -x^p \in \mathcal{P}$). This is called the Frobenius automorphism. This automorphism may change deppending on our choice of $\mathcal{P}$, but if we choose a different maximal ideal $\mathcal{P}$ containing $pO_K$, the new Frobenius is always a conjugate $\sigma \pi \sigma^{-1}$  of the other Frobenius, where $\sigma \in Aut(K/\Q)$ (this is because all other maximal ideals $\mathcal{P'}$ containing $p O_K$ are of the shape $\mathcal{P'} = \sigma(\mathcal{P})$). We losely refer to the Frobenius automorphism as $\pi_p$, keeping in mind that if the maximal ideal $\mathcal{P}$ is not specified $\pi_p$ is only defined up to conjugacy.

Due to this conjugacy ambiguity, it only makes sense to ask if $\pi_p$ belongs to a set $C \subset G$ if that set is closed under conjugation (that is, an union of conjugacy classes). We now state Chebotarev's density theorem, which describes the distribution of the Frobenius $\pi_p$ in $G$ as $p$ varies.

\begin{theorem}[Chebotarev's density theorem \cite{cheb}]
    With the set up above, for each conjugacy class $C$ of $G$, there exist infinitely many primes $p$ such that $\pi_p \in C$. In fact the proportion of primes $p$ satisfying $\pi_p \in C$ is $\frac{|C|}{|G|}$, that is
    $$\lim_{x \rightarrow \infty} \frac{\#\{p \le x: \pi_p \in C\}}{\pi(x)} = \frac{|C|}{|G|}.$$
\end{theorem}

A modern proof can be found in \cite{LA}. We won't be concerned with its proof on this paper, only with how to apply it. The following lemma explains the connection of the Frobenius automorphism with the roots mod $p$ of a polynomial

\begin{lemma} \label{fixedpoints}
    Let $p$ be a prime that does not divide the discriminant, leading coefficient or denominators of any coefficients of $f\in \Q[x]$. Then the number of roots of $f(x)$ mod $p$ is equal to the number of roots $\alpha \in K$ of $f(x)$ satisfying $\pi_p(\alpha) = \alpha$.
\end{lemma}

Notice that the number of fixed roots of $\pi_p$ does not depend on the conjugate of $\pi_p$ we choose, since a root $\alpha$ is fixed by $\pi_p$ iff the root $\sigma(\alpha)$ is fixed by $\sigma \pi_p \sigma^{-1}$. However in the proof below we work with a fixed maximal ideal $\mathcal{P}$, which fixes the conjugate of $\pi_p$ it defines.

The intuition behind the lemma is that a polynomial of degree $n$ always has all $n$ roots over some extension of $\mathbb{F}_p$, which can be thought as the roots over $K$ mapped by reduction mod $\mathcal{P}$ the ``residue field" $O_K/\mathcal{P}$, a finite field extension of $\F_p$ - but the ones that belong to the base field $\mathbb{F}_p$ are exactly the ones satisfying $\alpha^p \equiv \alpha$ - the ones fixed by Frobenius. For simplicity, we will prove it only in the case where $f$ is monic with integer coefficients.

\begin{proof}

Let $\alpha_1,...,\alpha_n \in K$ be the roots of $f$, which are in $O_K$ due to the assumption $f$ is monic with integer coefficients. We begin by noticing that if $p$ doesn't divide the discriminant of $f$, then $p \nmid D = \prod_{i \neq j}(\alpha_i-\alpha_j) \Rightarrow D = \prod_{i \neq j}(\alpha_i-\alpha_j) \notin \mathcal{P}$ (because $D$ is an integer and $\mathcal{P}\cap \mathbb{Z} = p \mathbb{Z}$), which means $\alpha_i -\alpha_j \notin \mathcal{P}$ for $i \neq j$. In other words, all the roots of $f$ are distinct mod $\mathcal{P}$.

Each solution $m$ of $f(m) \equiv 0$ mod $p$ must be equivalent to some $\alpha_i$ mod $\mathcal{P}$. This is because if $m$ is an integer, $p|f(m) \Rightarrow f(m) \in \mathcal{P} \Rightarrow (m-\alpha_1)...(m-\alpha_n) \in \mathcal{P} \Rightarrow m - \alpha_i \in \mathcal{P}$ for some $i$ (since $\mathcal{P}$ is maximal). Conversely, if $\alpha_i \equiv m$ mod $\mathcal{P}$ for some integer $m$, then $f(m) \equiv f(\alpha_i) = 0$ mod $\mathcal{P}$, which implies $f(m) \equiv 0$ mod $p$, since $\mathcal{P}\cap \mathbb{Z} = p \mathbb{Z}$. This means that to find out how many roots mod $p$ that $f$ has we just need to count how many of $\alpha_1,..., \alpha_n$ are equivalent to an integer mod $\mathcal{P}$.

If $\alpha_i \equiv m$ mod $\mathcal{P}$ for an some $m \in \mathbb{Z}$, then $\pi_p(\alpha_i) \equiv \alpha_i^p \equiv m^p \equiv m \equiv \alpha_i$ mod $\mathcal{P}$ (the fact that $m^p-m \in p O_K \subset \mathcal{P}$ follows from Fermat's little theorem, since $m$ is an integer). This implies $\pi_p(\alpha_i) = \alpha_i$, since an automorphism must send a root of $f$ to another root (since $\sigma(f(x)) = f(\sigma(x))$), and we established the only root congruent to $\alpha_i$ mod $\mathcal{P}$ is $\alpha_i$ itself. Conversely, if $\pi_p(\alpha_i) = \alpha_i$, we obtain $\alpha_i^p \equiv \alpha_i$ mod $\mathcal{P} \Rightarrow \alpha_i^p-\alpha_i \equiv \alpha_i(\alpha_i-1)...(\alpha_i-(p-1)) \equiv 0$ mod $\mathcal{P}$, which implies $\alpha_i \equiv m$ mod $\mathcal{P}$ for some $m \in \{0,...,p-1\}$.

So the number of solutions of $\pi_p(\alpha_i) = \alpha_i$ counts how many $\alpha_i$'s are congruent to some integer mod $\mathcal{P}$, completing the proof.

\end{proof}

We also sketch a proof of existence and uniqueness of the Frobenius automorphism, relying on the fact that any finite extension $K/\Q$ can be made of the shape $K = \Q[\beta]$ for some $\beta \in K$ (that's the primitive element theorem, see 7.3 of \cite{artingalois} for proof). For simplicity we also assume the minimal polynomial of $\beta$ is monic with integer coefficients, and that $O_K = \Z[\beta]$ (this can't always be arranged, but gives a good idea of where the Frobenius automorphism comes from).

\begin{proof}

Let the minimal polynomial $g$ of $\beta$ have roots $\beta = \beta_1$ and $\beta_2,...,\beta_n$. For $K = \Q[\beta]$, each of the $n$ automorphisms of $K/\Q$ is given by the map $\sigma(\beta) = \beta_i$ for some $i$ (which fully defines $\sigma$ on $K = \Q[\beta]$ - these other roots all belong to $K$ as well due to $K$ being a Galois extension of $\Q$).

We know from a binomial expansion that $g(\beta^p) \equiv g(\beta)^p = 0$ mod $\mathcal{P}$, that is $f(\beta^p)=(\beta^p-\beta_1)...(\beta^p-\beta_n) \in \mathcal{P} \Rightarrow \beta^p \equiv \beta_i$ mod $\mathcal{P}$ for some $\beta_i$. If $\sigma$ is the automorphism sending $\beta$ to $\beta_i$ we obtain for any integers $a_0,...,a_k$:
$$\sigma(a_0+a_1\beta+...+a_k \beta^k) = a_0+a_1 \beta_i+...+a_k \beta_i^k,$$ which is $$\equiv a_0+a_1\beta^p+...+a_k \beta^{pk} \equiv (a_0+a_1\beta+...+a_k \beta^k)^p\mod \mathcal{P},$$
so $\sigma(x) \equiv x^p \mod \mathcal{P}$ for every $x \in \Z[\beta] = O_K$, proving existence.

Uniqueness follows from the fact that if $p$ doesn't divide the discriminant of $g$ then the roots $\beta_1,..., \beta_n$ are distinct mod $\mathcal{P}$, so knowing $\sigma(\beta) \equiv \beta^p$ mod $\mathcal{P}$ determines what root $\sigma(\beta)$ is sent to, which fully determines $\sigma$. 
    
\end{proof}

We made a number of simplifying assumptions here to get the ideas across and cover extensions of $\Q$ in a minimal way. For a more detailed and general approach to the Frobenius automorphism of number field extensions, see chapter 6 of \cite{samuel}.

This set up is very useful in practice for studying root counts of a polynomial mod $p$. Let us see two examples, diving deeper in the polynomials we provided earlier.

\begin{example}
    $f(x) = (x^2+1)(x^2+2)(x^2-2)$ has roots $\pm i, \pm i \sqrt{2}$ and $\pm \sqrt{2}$. Its splitting field is $K = \Q[i, \sqrt{2}]$, a degree 4 extension of $\mathbb{Q}$. Any automorphism of $K$ must permute the roots of $g(x) = x^2+1$ (that is due to $g(\sigma(\alpha)) = \sigma(g(\alpha)) = 0$ for a root $\alpha$ of $g$), and the roots of $x^2-2$ as well. That is, $\sigma(i) = \pm i$ and $\sigma(\sqrt{2}) = \pm \sqrt{2}$. The value of $\sigma$ at these two roots determines the value of $\sigma(i \sqrt{2}) = \sigma(i) \sigma(\sqrt{2})$, fully determining the automorphism. With this, one can see that $\sigma \mapsto (\frac{\sigma(i)}{i}, \frac{\sigma(\sqrt{2})}{\sqrt{2}})$ identifies $Aut(K/\Q)$ as the group $\{\pm 1\} \times \{\pm 1\}$.

    The identity automorphism $(1,1)$ fixes all $6$ roots of $f$, and each of the other automorphisms fixes two roots of $f$ ($(1, -1)$ fixes $\pm i$, $(-1,1)$ fixes $\pm \sqrt{2}$ and $(-1,-1)$ fixes $\pm i \sqrt{2}$).

    Due to commutativity, each of the 4 elements of $Aut(K/\Q)$ is alone in a conjugacy class. We conclude by Chebotarev's theorem and Lemma \ref{fixedpoints} that for a proportion $\frac{1}{4}$ of primes, $f(x)$ has exactly 6 roots mod $p$, and for a proportion $\frac{3}{4}$ of primes it has exactly 2 roots mod $p$. These are the only possibilities for the root count of $f$ mod $p$ for a prime that doesn't divide its discriminant ($p \neq 2,3$).

    Which of the cases the Frobenius belongs to is determined by $(\frac{\pi_p(i)}{i}, \frac{\pi_p(\sqrt{2})}{\sqrt{2}}) \equiv ((-1)^{\frac{p-1}{2}},2^{\frac{p-1}{2}}) \equiv ((\frac{-1}{p}), (\frac{2}{p})) \mod p$.
\end{example}

\begin{example}
    $f(x) = (x^2+x+1)(x^3-2)$ has roots $\omega, \omega^2, \sqrt[3]{2}, \omega\sqrt[3]{2}$ and $\omega^2\sqrt[3]{2}$, where $\omega$ is a cube root of unity. The splitting field of $f$ is therefore $K = \Q[\sqrt[3]{2},\omega]$, a degree $6$ extension of $\Q$. An automorphism $\sigma$ permutes the roots of $g(x) = x^3-2$. With relations such as $\sigma(\omega) = \frac{\sigma(\omega\sqrt[3]{2})}{\sigma(\sqrt[3]{2})}$, one can use the action of $\sigma$ on the $3$ roots of $x^3-2$ to determine the action of $\sigma$ on all roots of $f$, fully determining $\sigma$. It can be verified each of those permutations of the roots of $x^3-2$ in fact provides an automorphism of $K/\Q$, so $Aut(K/\Q) \simeq S_3$.
    
    We then have $3$ cases depending on the conjugacy class of an automorphism $\sigma \in S_3$:

    \begin{itemize}

    \item If $\sigma$ fixes the $3$ roots of $x^3-2$, then it fixes all $5$ roots of $f$. From lemma \ref{fixedpoints}, if $\pi_p$ is an automorphism of this type, $f$ has $5$ roots mod $p$.

    \item If $\sigma$ acts on the 3 roots of $x^3-2$ like a permutation of cycle type (2)(1) in $S_3$. In this case one can deduce $\sigma$  flips $\omega$ and $\omega^2$. So in total $\sigma$ fixes a single root of $f$ in this case. From lemma \ref{fixedpoints}, if $\pi_p$ is an automorphism of this type, $f$ has $1$ root mod $p$.

    \item If  $\sigma$ is of cycle type (3) in $S_3$, one deduces it must fix $\omega$ and $\omega^2$. In this case $\sigma$ fixes 2 roots of $f$. From lemma \ref{fixedpoints}, if $\pi_p$ is an automorphism of this type, $f$ has $2$ roots mod $p$.

    \end{itemize}

    Hence for a prime $p$ not dividing the discriminant of $f$ ($p \neq 2,3$), $f$ will always have 5, 1 or 2 roots mod $p$. Furthermore, because $S_3$ has $1$ identity, $3$ elements of cycle type $(1)(2)$ (which form a conjugacy class), and $2$ elements of cycle type $(3)$ (another conjugacy class), we conclude from Chebotarev's theorem that the density of the set of primes in each of these cases are $\frac{1}{6}$, $\frac{3}{6}$ and $\frac{2}{6}$ respectively.
    
    Whether $\pi_p$ fixes $\omega$ or flips it with $\omega^2$ is determined by the value of $p$ mod $3$  (since $\pi_p(\omega) = \omega^p$), and in the case $p \equiv 1$ mod $3$ whether $\pi_p$ fixes $\sqrt[3]{2}$ is determined by whether $2^{\frac{p-1}{3}} \equiv 1$ mod $p$. This information fully determines the class of $\pi_p$.

\end{example}

\section{Proofs of Theorem~\ref{real root} and Theorem~\ref{positive}}

We are ready to prove the following result, which has Theorem \ref{real root} as a corollary.

\begin{theorem} \label{squarefree}
    Let $f \in \Q[x]$ be a polynomial with distinct complex roots, such that $f$ has at least $k$ roots mod $p$ for every large prime $p$. Then $f$ has at least $k$ real roots.
\end{theorem}

The two polynomials we provided are examples of the sharpness of this statement, for $k = 2$ and $k =1$ respectively.

\begin{proof}

Let $\alpha_1,...,\alpha_n \in K \subset \mathbb{C}$ be the complex roots of $f$.
Complex conjugation provides an automorphism of $K/\mathbb{Q}$ - name this automorphism $\theta$. By Chebotarev's density theorem, there exist infinitely many $p$ such that $\pi_p$ is conjugate to $\theta$, say $\pi_p = \sigma^{-1} \theta \sigma$ for some automorphism $\sigma \in Aut(K/\Q)$. By assumption, $f$ has at least $k$ roots mod $p$ for such $p$ - by Lemma \ref{fixedpoints} this implies at least $k$ of $\alpha_1,..., \alpha_n$ satisfy $\pi_p(\alpha_i) = \alpha_i$, which implies each of those satisfies $\sigma^{-1} \theta \sigma(\alpha_i) = \alpha_i \Rightarrow \theta(\sigma(\alpha_i)) = \sigma(\alpha_i)$. But $\theta$ is complex conjugation, so this implies each of those $\sigma(\alpha_i)$ is real. Because $\sigma$ is an automorphism it permutes the roots, so $\sigma(\alpha_i)$ are $k$ distinct roots of $f$ - we conclude $f$ has at least $k$ real roots, as desired. The assumption that $f$ has distinct roots is required in the application of Lemma \ref{fixedpoints} to make sure the discriminant of $f$ is non-zero, so large primes don't divide it.

\end{proof}

Notice also that by using the density from Chebotarev's theorem, at least $\frac{1}{|Aut(K/\mathbb{Q})|}$ of primes have $\pi_p$ conjugate to complex conjugation, so the assumption that $f$ has roots mod $p$ for every large $p$ may be weakened to $f$ has $k$ roots mod $p$ for a proportion greater than $1-\frac{1}{|Aut(K/\mathbb{Q})|}$ of primes - this is enough to force one of those primes to have $\pi_p$ conjugate to complex conjugation, and the rest of the proof works equally.

\begin{proof}[Proof of Theorem~\ref{real root}]

If $f$ has a root mod $p$ for every $p$, the product of its irreducible factors over $\Q[x]$ also has this property. Applying Theorem \ref{squarefree} to this polynomial with $k = 1$, we obtain it has a real root, which is also a root of $f$.
    
\end{proof}

\begin{proof}[Proof of Theorem~\ref{positive}]
For a large enough prime $p$, if the equation $g(x,y) =  ax^2+bxy+cy^2 \equiv 0$ has a non-trivial solution, then certainly $y \neq 0$ mod $p$, since if $y \equiv 0$ mod $p$ we obtain $ax^2 \equiv 0$ mod $p$, which assuming $p$ is large compared with $a>0$ would imply $x \equiv 0$ mod $p$ as well, contradicting that $(x,y)$ is non-trivial. Dividing by $y^2$ we obtain that $at^2+bt+c \equiv 0 $ mod $ p$ has a root  mod $p$, namely $t = xy^{-1}$.  So if one of the forms $a_ix^2+b_ixy+c_iy^2$ always has a non-trivial solution mod $p$, then
\[
f(t) = \prod_{i=1}^{n} (a_i t^2 + b_i t+c_i)
\]
has a root mod $p$ for every large prime $p$, which would imply by Theorem~\ref{real root} that $a_it^2+b_it+c_i$ has a real root for some $i$, i.e., it is not positive definite. This contradiction finishes the proof.

\end{proof}

Incorporating density in this argument we see that this product of quadratic polynomials fails to have a root mod $p$ for a proportion at least $\frac{1}{|Aut(K/\mathbb{Q})|}$ of primes (namely, any prime whose Frobenius is complex conjugation, which won't fix any of the roots of $f$, which are all non-real if all forms are positive definite). This is least $\frac{1}{2^n}$, because the splitting field of $f(t) = \prod_{i=1}^{n} (a_i t^2 + b_i t+c_i)$ is $K = K_n$ where $\mathbb{Q} = K_0 \subset K_1 \subset ... \subset K_n$, and each $K_{i}$ is obtained by including the roots of $a_it^2+b_it+c_i$ to $K_{i-1}$. So each of field extension $K_{i}/K_{i-1}$ has degree at most 2, which makes the degree of $K/\mathbb{Q}$ at most $2^n$ (in fact the Galois group of this extension is $\{\pm1\}^m$ for some $m \le n$). So a set of $n$ positive definite quadratic forms must fail to cover a proportion of at least $\frac{1}{2^n}$ of prime numbers.


\bibliographystyle{plain}
\bibliography{gold}{}
\end{document}